\documentclass[11pt]{article}

\usepackage[a4paper,margin=1in]{geometry}
\usepackage{amsmath,amssymb,amsthm,mathtools}
\usepackage{enumitem}
\usepackage{hyperref}
\usepackage{microtype}
\usepackage{tikz}

\hypersetup{
    colorlinks=true,
    linkcolor=blue,
    citecolor=red,
    urlcolor=blue
}

\newtheorem{theorem}{Theorem}[section]
\newtheorem{lemma}[theorem]{Lemma}
\newtheorem{proposition}[theorem]{Proposition}
\newtheorem{conjecture}[theorem]{Conjecture}

\theoremstyle{definition}

\newcommand{\per}{\operatorname{per}}
\newcommand{\M}{\mathcal M}

\title{Counterexamples to a Conjecture on Laplacian Ratios of Trees}
\author{
Priyanshu Pant\thanks{Indian Institute of Technology Indore; website: \url{https://sites.google.com/view/priyanshupant}.}\\
\href{mailto:priyanshupant03@gmail.com}{priyanshupant03@gmail.com}
}
\date{}

\begin{document}

\maketitle

\begin{abstract}
For a graph \(G\) with no isolated vertices, its Laplacian ratio is defined as
\[
    \pi(G)=\frac{\per(L(G))}{\prod_{v\in V(G)} d(v)},
\]
where \(L(G)\) is the Laplacian matrix of \(G\), \(d(v)\) is the degree of
\(v\), and \(\per\) denotes the permanent.  Brualdi and Goldwasser
\cite{BG1984} asked for the maximum value of \(\pi(T)\) among trees \(T\) with
a fixed number of vertices.  Wu, Dong and Lai \cite{WDL2025} recently proposed
a conjectural answer to this problem.  We give infinite families of
counterexamples to their conjecture.
\end{abstract}


\noindent\textbf{Keywords.}
Laplacian matrix, permanent, tree, matching, extremal graph.

\section{Introduction}

All graphs in this paper are finite and simple.  For a graph \(G\), we write
\(V(G)\) and \(E(G)\) for its vertex set and edge set.  For a vertex
\(v\in V(G)\), let \(d(v)\) denote its degree.  We write \(A(G)\) for the
adjacency matrix of \(G\), and \(D(G)\) for the diagonal matrix whose
\(v\)-th diagonal entry is \(d(v)\).  The Laplacian matrix of \(G\) is
\[
    L(G)=D(G)-A(G).
\]
For an \(n\times n\) matrix \(M=(m_{ij})\), its permanent is
\[
    \per(M)
    =
    \sum_{\sigma\in S_n}\prod_{i=1}^n m_{i,\sigma(i)},
\]
where \(S_n\) is the symmetric group on \(\{1,\ldots,n\}\).
If \(G\) has no isolated vertices, its Laplacian ratio is defined by
\[
    \pi(G)
    =
    \frac{\per(L(G))}
    {\prod_{v\in V(G)} d(v)}.
\]

Brualdi and Goldwasser \cite{BG1984} initiated the study of permanents of
Laplacian matrices of trees and bipartite graphs.  For trees, they proved the
sharp bounds
\[
    2(n-1) \le \per(L(T)) \le \per(L(P_n)),
\]
where \(T\) is a tree on \(n\) vertices and \(P_n\) is the path on \(n\)
vertices.  The left equality holds if and only if \(T\) is a star, and the
right equality holds if and only if \(T\) is a path.  The upper bound can be
written explicitly as
\[
    \per(L(P_n))
    =
    \left(1-\frac{\sqrt2}{2}\right)(1+\sqrt2)^n
    +
    \left(1+\frac{\sqrt2}{2}\right)(1-\sqrt2)^n .
\]
Brualdi and Goldwasser also proved that every tree \(T\) satisfies
\[
    \pi(T)\ge 2,
\]
with equality if and only if \(T\) is a star.  They then asked for extremal
results for the Laplacian ratio of trees.  In particular, they asked for the
maximum value of \(\pi(T)\) among all trees with a fixed number of vertices.

Wu, Dong and Lai \cite{WDL2025} recently studied this maximum problem and
proposed the following conjecture.

\begin{conjecture}[Wu--Dong--Lai, Conjecture 1.2 \cite{WDL2025}]
Let \(T\) be a tree on \(n\) vertices.
\begin{enumerate}[label=\textup{(\roman*)}]
    \item If \(n\) is odd, then
    \[
        \pi(T) \le 2\left(\frac32\right)^{(n-3)/2},
    \]
    with equality if and only if
    \[
        T=S\left(n,\frac{n-1}{2}\right).
    \]

    \item If \(n=4k\), then
    \[
        \pi(T)
        \le
        \frac{4n^2+8n+24}{n(n+4)}
        \left(\frac32\right)^{(n-6)/2},
    \]
    with equality if and only if
    \[
        T=S\left(n,\frac{n+2}{2},\frac{n-2}{2}\right).
    \]

    \item If \(n=4k+2\), then
    \[
        \pi(T)
        \le
        \frac{4n^2+8n+40}{(n+2)^2}
        \left(\frac32\right)^{(n-6)/2},
    \]
    with equality if and only if
    \[
        T=S\left(n,\frac n2,\frac n2\right).
    \]
\end{enumerate}
\end{conjecture}

Here \(S(n,(n-1)/2)\) is the subdivided star obtained from the star on
\((n+1)/2\) vertices by attaching one pendant edge to each noncentral vertex.
The tree \(S(n,a,b)\) is obtained from two such subdivided stars with \(a\) and
\(b\) vertices by adding an edge between their centers.

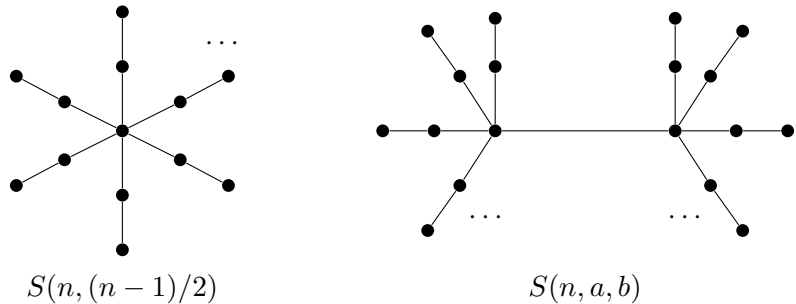
\begin{figure}[h]
\centering
\begin{tikzpicture}[
    scale=0.85,
    vertex/.style={circle, fill=black, inner sep=1.7pt},
    labelnode/.style={rectangle, fill=none, inner sep=1pt}
]
    \node[vertex, label=below:\(\)] (c) at (0,0) {};
    \node[vertex] (u1) at (0,1.0) {};
    \node[vertex] (v1) at (0,1.85) {};
    \node[vertex] (u2) at (0.9,0.45) {};
    \node[vertex] (v2) at (1.65,0.85) {};
    \node[vertex] (u3) at (0.9,-0.45) {};
    \node[vertex] (v3) at (1.65,-0.85) {};
    \node[vertex] (u4) at (0,-1.0) {};
    \node[vertex] (v4) at (0,-1.85) {};
    \node[vertex] (u5) at (-0.9,-0.45) {};
    \node[vertex] (v5) at (-1.65,-0.85) {};
    \node[vertex] (u6) at (-0.9,0.45) {};
    \node[vertex] (v6) at (-1.65,0.85) {};
    \draw (c)--(u1)--(v1) (c)--(u2)--(v2) (c)--(u3)--(v3)
          (c)--(u4)--(v4) (c)--(u5)--(v5) (c)--(u6)--(v6);
    \node[labelnode] at (1.55,1.35) {\(\cdots\)};
    \node[labelnode] at (0,-2.45) {\(S(n,(n-1)/2)\)};

    \begin{scope}[xshift=5.8cm]
        \node[vertex, label=below:\(\)] (cx) at (0,0) {};
        \node[vertex, label=below:\(\)] (cy) at (2.8,0) {};
        \draw (cx)--(cy);
        \node[vertex] (xu1) at (-0.55,0.85) {};
        \node[vertex] (xv1) at (-1.05,1.55) {};
        \node[vertex] (xu2) at (-0.95,0) {};
        \node[vertex] (xv2) at (-1.75,0) {};
        \node[vertex] (xu3) at (-0.55,-0.85) {};
        \node[vertex] (xv3) at (-1.05,-1.55) {};
        \node[vertex] (xu4) at (0,1.0) {};
        \node[vertex] (xv4) at (0,1.75) {};
        \node[vertex] (yu1) at (3.35,0.85) {};
        \node[vertex] (yv1) at (3.85,1.55) {};
        \node[vertex] (yu2) at (3.75,0) {};
        \node[vertex] (yv2) at (4.55,0) {};
        \node[vertex] (yu3) at (3.35,-0.85) {};
        \node[vertex] (yv3) at (3.85,-1.55) {};
        \node[vertex] (yu4) at (2.8,1.0) {};
        \node[vertex] (yv4) at (2.8,1.75) {};
        \draw (cx)--(xu1)--(xv1) (cx)--(xu2)--(xv2)
              (cx)--(xu3)--(xv3) (cx)--(xu4)--(xv4)
              (cy)--(yu1)--(yv1) (cy)--(yu2)--(yv2)
              (cy)--(yu3)--(yv3) (cy)--(yu4)--(yv4);
        \node[labelnode] at (-0.15,-1.35) {\(\cdots\)};
        \node[labelnode] at (2.95,-1.35) {\(\cdots\)};
        \node[labelnode] at (1.4,-2.45) {\(S(n,a,b)\)};
    \end{scope}
\end{tikzpicture}
\caption{The trees \(S(n,(n-1)/2)\) and \(S(n,a,b)\) appearing in the
Wu--Dong--Lai conjecture.}
\label{fig:wdl-trees}
\end{figure}

We show that this conjecture is false in all three cases.  The counterexamples
come from a family of trees \(T(a_1,\ldots,a_m)\), defined in
Section~\ref{sec:pathcore}.  For this family we obtain an explicit formula
\[
    \pi(T(a_1,\ldots,a_m))
    =
    \left(\frac32\right)^{a_1+\cdots+a_m} f_m,
\]
where \(f_m\) is given by a simple two-term recurrence.  Applying this formula
to the subfamilies
\[
    T(3,t,3),\qquad T(t,t,t,t),\qquad T(t,t,t+1,t)
    \quad (t\ge4),
\]
gives strict counterexamples to the three proposed upper bounds.  We also show
that the equality statement in the odd case already fails at \(n=21\).

\section{The family \texorpdfstring{\(T(a_1,\ldots,a_m)\)}{T(a1,...,am)}}
\label{sec:pathcore}

We now introduce the family of trees used for the counterexamples.  Let
\(m\ge2\), and let \(a_1,\ldots,a_m\) be nonnegative integers.  Start with the
path
\[
    x_1-x_2-\cdots-x_m .
\]
For each \(i\), attach \(a_i\) pendant paths of length two to the vertex
\(x_i\).  We denote the resulting tree by
\[
    T(a_1,\ldots,a_m).
\]
Thus each attached path has the form
\[
    x_i-y-z,
\]
where \(z\) is a leaf.  We call the path \(x_1x_2\cdots x_m\) the core path.

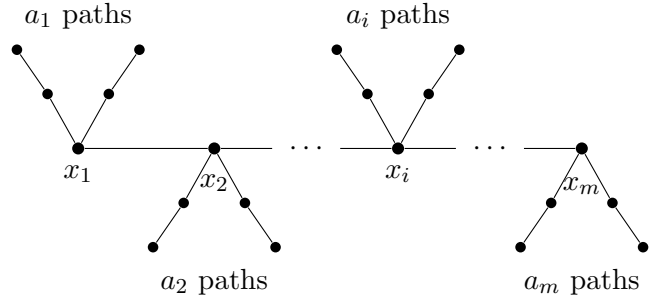
\begin{figure}[h]
\centering
\begin{tikzpicture}[
    scale=0.9,
    vertex/.style={circle, fill=black, inner sep=1.6pt},
    smallvertex/.style={circle, fill=black, inner sep=1.4pt},
    labelnode/.style={rectangle, fill=none, inner sep=1pt}
]

\node[vertex, label=below:\(x_1\)] (x1) at (0,0) {};
\node[vertex] (x2) at (2,0) {};
\node[labelnode] at (2,-0.5) {\(x_2\)};
\node[labelnode] at (3.35,0) {\(\cdots\)};
\node[vertex, label=below:\(x_i\)] (xi) at (4.7,0) {};
\node[labelnode] at (6.05,0) {\(\cdots\)};
\node[vertex] (xm) at (7.4,0) {};
\node[labelnode] at (7.4,-0.55) {\(x_m\)};

\draw (x1)--(x2);
\draw (x2)--(2.85,0);
\draw (3.85,0)--(xi);
\draw (xi)--(5.55,0);
\draw (6.55,0)--(xm);

\node[smallvertex] (x1a1) at (-0.45,0.8) {};
\node[smallvertex] (x1b1) at (-0.9,1.45) {};
\node[smallvertex] (x1a2) at (0.45,0.8) {};
\node[smallvertex] (x1b2) at (0.9,1.45) {};
\draw (x1)--(x1a1)--(x1b1);
\draw (x1)--(x1a2)--(x1b2);
\node[labelnode] at (0,1.95) {\(a_1\) paths};

\node[smallvertex] (x2a1) at (1.55,-0.8) {};
\node[smallvertex] (x2b1) at (1.1,-1.45) {};
\node[smallvertex] (x2a2) at (2.45,-0.8) {};
\node[smallvertex] (x2b2) at (2.9,-1.45) {};
\draw (x2)--(x2a1)--(x2b1);
\draw (x2)--(x2a2)--(x2b2);
\node[labelnode] at (2,-1.95) {\(a_2\) paths};

\node[smallvertex] (xia1) at (4.25,0.8) {};
\node[smallvertex] (xib1) at (3.8,1.45) {};
\node[smallvertex] (xia2) at (5.15,0.8) {};
\node[smallvertex] (xib2) at (5.6,1.45) {};
\draw (xi)--(xia1)--(xib1);
\draw (xi)--(xia2)--(xib2);
\node[labelnode] at (4.7,1.95) {\(a_i\) paths};

\node[smallvertex] (xma1) at (6.95,-0.8) {};
\node[smallvertex] (xmb1) at (6.5,-1.45) {};
\node[smallvertex] (xma2) at (7.85,-0.8) {};
\node[smallvertex] (xmb2) at (8.3,-1.45) {};
\draw (xm)--(xma1)--(xmb1);
\draw (xm)--(xma2)--(xmb2);
\node[labelnode] at (7.4,-1.95) {\(a_m\) paths};

\end{tikzpicture}
\caption{The tree \(T(a_1,\ldots,a_m)\).  The core is the path
\(x_1-x_2-\cdots-x_m\), and \(a_i\) pendant paths of length two are attached to
\(x_i\).}
\label{fig:T-family}
\end{figure}

The number of vertices is
\[
    |V(T(a_1,\ldots,a_m))|
    =
    m+2(a_1+\cdots+a_m).
\]
The degrees of the core vertices are
\[
    d(x_i)=
    \begin{cases}
        a_i+1, & i=1 \text{ or } i=m,\\
        a_i+2, & 1<i<m.
    \end{cases}
\]
All other vertices have degree \(1\) or \(2\).

We now compute the Laplacian ratio of \(T(a_1,\ldots,a_m)\).  The computation
is based on matchings.  Recall that a matching in a graph is a set of edges no
two of which share a common endpoint.  We write \(\M(T)\) for the set of all
matchings of a tree \(T\), including the empty matching.  If \(M\) is a
matching, then \(V(M)\) denotes the set of vertices covered by the edges of
\(M\).

\begin{lemma}\label{lem:matching}
Let \(T\) be a tree with no isolated vertices, and label its vertices by
\(1,\ldots,n\).  Then
\[
    \pi(T)
    =
    \sum_{M\in \M(T)}
    \prod_{uv\in M}\frac{1}{d(u)d(v)}.
\]
\end{lemma}

\begin{proof}
By the definition of the permanent,
\[
    \per(L(T))
    =
    \sum_{\sigma\in S_n} \prod_{i=1}^n L(T)_{i,\sigma(i)}.
\]
We examine which permutations \(\sigma\) give a nonzero term.  For the
Laplacian matrix of \(T\),
\[
    L(T)_{i,i}=d(i),
    \qquad
    L(T)_{i,j}=
    \begin{cases}
        -1, & ij\in E(T),\\
        0, & ij\notin E(T),
    \end{cases}
    \quad (i\ne j).
\]
Thus, if a term is nonzero, then whenever \(\sigma(i)\ne i\), the vertices
\(i\) and \(\sigma(i)\) must be adjacent in \(T\).
Now write \(\sigma\) as a product of disjoint cycles.  Suppose one of these
cycles has length at least three, say
\[
    (v_1\,v_2\,\cdots\,v_r),
    \qquad r\ge3.
\]
For the corresponding term to be nonzero, all edges
\[
    v_1v_2,\ v_2v_3,\ \ldots,\ v_{r-1}v_r,\ v_rv_1
\]
would have to belong to \(T\).  These edges form a cycle in \(T\), which is
impossible because \(T\) is a tree.  Therefore every nonzero term comes only
from fixed points and disjoint transpositions.

The transpositions in such a permutation are disjoint edges of \(T\), hence
they form a matching \(M\).  Conversely, every matching \(M\) gives one such
permutation: transpose the endpoints of each edge in \(M\), and fix all other
vertices.
For a transposition \(uv\in M\), the contribution is
\[
    L(T)_{u,v}L(T)_{v,u}=(-1)(-1)=1.
\]
Every vertex not covered by \(M\) is fixed and contributes its diagonal entry
\(d(v)\).  Hence the contribution of the matching \(M\) to
\(\per(L(T))\) is
\[
    \prod_{v\notin V(M)} d(v).
\]
Therefore
\[
    \per(L(T))
    =
    \sum_{M\in \M(T)}
    \prod_{v\notin V(M)} d(v).
\]
Dividing by \(\prod_{v\in V(T)}d(v)\), we get
\[
    \pi(T)
    =
    \sum_{M\in \M(T)}
    \frac{\prod_{v\notin V(M)} d(v)}
         {\prod_{v\in V(T)} d(v)}
    =
    \sum_{M\in \M(T)}
    \prod_{uv\in M}\frac{1}{d(u)d(v)}.
\]
\end{proof}

We now apply the matching formula to the family \(T(a_1,\ldots,a_m)\).

\begin{proposition}\label{prop:recurrence}
Let \(m\ge2\), and put \(d_i=d_{T(a_1,\ldots,a_m)}(x_i)\).  Define
\(f_0=1\), \(f_1=1+a_1/(3d_1)\), and, for \(2\le i\le m\),
\[
    f_i=\left(1+\frac{a_i}{3d_i}\right)f_{i-1}
        +\frac{1}{d_{i-1}d_i}f_{i-2}.
\]
Then
\[
    \pi(T(a_1,\ldots,a_m))
    =
    \left(\frac32\right)^{a_1+\cdots+a_m}f_m.
\]
\end{proposition}
\begin{proof}
We use Lemma~\ref{lem:matching}.  In the weighted matching formula, each edge
\(uv\) has weight \(1/(d(u)d(v))\).
First consider one pendant path of length two attached to \(x_i\):
\[
    x_i-y-z,
\]
where \(z\) is a leaf.  The edge \(yz\) has weight
\[
    \frac{1}{d(y)d(z)}=\frac12.
\]
If \(x_i\) is already matched to a core edge, then the edge \(x_iy\) cannot be
used.  Thus this pendant path contributes either the empty matching or the edge
\(yz\), so its contribution is
\[
    1+\frac12=\frac32.
\]
Now suppose \(x_i\) is not matched to a core edge.  There are two possibilities
for the \(a_i\) pendant paths attached to \(x_i\).  If \(x_i\) is not matched
to any of them, then each pendant path contributes \(3/2\), giving
\[
    \left(\frac32\right)^{a_i}.
\]
If \(x_i\) is matched to one of them, then there are \(a_i\) choices for the
chosen path.  The chosen edge has weight
\[
    \frac{1}{2d_i},
\]
and the remaining \(a_i-1\) pendant paths each contribute \(3/2\).  This gives
\[
    a_i\cdot \frac{1}{2d_i}
    \left(\frac32\right)^{a_i-1}.
\]
Therefore, when \(x_i\) is not matched to a core edge, the total contribution
from the pendant paths at \(x_i\) is
\[
    \left(\frac32\right)^{a_i}
    +
    a_i\cdot \frac{1}{2d_i}
    \left(\frac32\right)^{a_i-1}
    =
    \left(\frac32\right)^{a_i}
    \left(1+\frac{a_i}{3d_i}\right).
\]
We now factor out
\[
    \left(\frac32\right)^{a_1+\cdots+a_m}
\]
from all pendant paths.  After this factor is removed, only a weighted matching
problem on the core path
\[
    x_1-x_2-\cdots-x_m
\]
remains.  In this reduced problem, a core vertex \(x_i\) not covered by
a chosen core edge contributes
\[
    1+\frac{a_i}{3d_i},
\]
and a chosen core edge \(x_{i-1}x_i\) contributes
\[
    \frac{1}{d_{i-1}d_i}.
\]
Let \(f_i\) be the reduced contribution from the first \(i\) core vertices.
Then \(f_0=1\), and
\[
    f_1=1+\frac{a_1}{3d_1}.
\]
For \(i\ge2\), there are two possibilities.  If \(x_i\) is not matched to
\(x_{i-1}\), the contribution is
\[
    \left(1+\frac{a_i}{3d_i}\right)f_{i-1}.
\]
If \(x_i\) is matched to \(x_{i-1}\), then the edge \(x_{i-1}x_i\) contributes
\[
    \frac{1}{d_{i-1}d_i},
\]
and the remaining contribution comes from the first \(i-2\) core vertices.
This gives
\[
    \frac{1}{d_{i-1}d_i}f_{i-2}.
\]
Hence
\[
    f_i=
    \left(1+\frac{a_i}{3d_i}\right)f_{i-1}
    +
    \frac{1}{d_{i-1}d_i}f_{i-2}.
\]
Finally, multiplying the reduced contribution \(f_m\) by the factored term
gives
\[
    \pi(T(a_1,\ldots,a_m))
    =
    \left(\frac32\right)^{a_1+\cdots+a_m}f_m.
\]
\end{proof}

\section{The counterexamples}

We now use Proposition~\ref{prop:recurrence} to obtain the three families of
counterexamples.  

\subsection{Odd order}

Let
\[
    A_t=T(3,t,3),\qquad t\ge3.
\]
Then
\[
    |V(A_t)|=3+2(3+t+3)=2t+15,
\]
so \(A_t\) has odd order.  The core degrees are
\[
    (d_1,d_2,d_3)=(4,t+2,4).
\]
Using Proposition~\ref{prop:recurrence}, we get
\[
    \pi(A_t)
    =
    \left(\frac32\right)^{t+6}
    \frac{5(5t+9)}{12(t+2)}.
\]
Put \(n=|V(A_t)|=2t+15\).  Then
\[
    \frac{n-3}{2}=t+6,
\]
so the conjectured odd upper bound is
\[
    2\left(\frac32\right)^{t+6}.
\]
Therefore
\[
\begin{aligned}
\pi(A_t)
-
2\left(\frac32\right)^{t+6}
&=
\left(\frac32\right)^{t+6}
\left(
\frac{5(5t+9)}{12(t+2)}-2
\right)  \\
&=
\left(\frac32\right)^{t+6}
\frac{t-3}{12(t+2)}.
\end{aligned}
\]
This is positive for every \(t\ge4\).  Hence the odd upper bound fails for
every odd \(n\ge23\).  The first example in this family is
\[
    A_4=T(3,4,3),
\]
which has \(23\) vertices.

There is also a failure of the equality statement at \(t=3\).  In this case
\(A_3=T(3,3,3)\) has \(21\) vertices, and the above difference is zero.  Thus
\(T(3,3,3)\) attains the conjectured upper bound.  However it is not
isomorphic to the conjectured equality case \(S(21,10)\).  Indeed,
\(T(3,3,3)\) has degree sequence
\[
    1^9,2^9,4,4,5,
\]
whereas \(S(21,10)\) has degree sequence
\[
    1^{10},2^{10},10.
\]
Thus the equality statement in the odd case already fails at \(n=21\).

\subsection{The case \texorpdfstring{\(n=4k\)}{n=4k}}

Let
\[
    B_t=T(t,t,t,t),\qquad t\ge4.
\]
Then
\[
    |V(B_t)|=4+2(4t)=8t+4,
\]
which is divisible by \(4\), since \(8t+4=4(2t+1)\).  The core degrees are
\[
    (d_1,d_2,d_3,d_4)=(t+1,t+2,t+2,t+1).
\]
Using Proposition~\ref{prop:recurrence}, we get
\[
\pi(B_t)
=
\left(\frac32\right)^{4t}
\frac{
2(128t^4+576t^3+1152t^2+1080t+405)
}{
81(t+1)^2(t+2)^2
}.
\]
Put \(n=|V(B_t)|=8t+4\).  The conjectured upper bound in the case \(n=4k\) is
\[
    \frac{4n^2+8n+24}{n(n+4)}
    \left(\frac32\right)^{(n-6)/2}.
\]
Since
\[
    \frac{n-6}{2}=4t-1,
\]
a direct simplification gives
\[
\begin{aligned}
&\pi(B_t)
-
\frac{4n^2+8n+24}{n(n+4)}
\left(\frac32\right)^{(n-6)/2}  \\
&\qquad =
\left(\frac32\right)^{4t}
\frac{
 t^2(160t^3-280t^2-1197t-873)
}{
162(t+1)^2(t+2)^2(2t+1)
}.
\end{aligned}
\]
The denominator is positive.  To see that the numerator is positive for
\(t\ge4\), write \(t=s+4\).  Then
\[
160t^3-280t^2-1197t-873
=
160s^3+1640s^2+4243s+99,
\]
which is positive for every \(s\ge0\).  Therefore the difference is positive
for every \(t\ge4\).  Hence the \(n=4k\) upper bound fails for every
\(n=8t+4\) with \(t\ge4\).  The first example in this family is
\[
    B_4=T(4,4,4,4),
\]
which has \(36\) vertices.

\subsection{The case \texorpdfstring{\(n=4k+2\)}{n=4k+2}}

Let
\[
    C_t=T(t,t,t+1,t),\qquad t\ge4.
\]
Then
\[
    |V(C_t)|=4+2(t+t+(t+1)+t)=8t+6.
\]
Since
\[
    |V(C_t)|=8t+6=4(2t+1)+2,
\]
this tree falls under the \(n\equiv2\pmod4\) case of the conjecture.
The core degrees are
\[
    (d_1,d_2,d_3,d_4)=(t+1,t+2,t+3,t+1).
\]
Using Proposition~\ref{prop:recurrence}, we get
\[
\pi(C_t)
=
\left(\frac32\right)^{4t+1}
\frac{
2(128t^4+704t^3+1536t^2+1512t+567)
}{
81(t+1)^2(t+2)(t+3)
}.
\]
Put \(n=|V(C_t)|=8t+6\).  The conjectured upper bound in the case \(n=4k+2\)
is
\[
    \frac{4n^2+8n+40}{(n+2)^2}
    \left(\frac32\right)^{(n-6)/2}.
\]
Since
\[
    \frac{n-6}{2}=4t,
\]
a direct simplification gives
\[
\begin{aligned}
&\pi(C_t)
-
\frac{4n^2+8n+40}{(n+2)^2}
\left(\frac32\right)^{(n-6)/2}  \\
&\qquad =
\left(\frac32\right)^{4t}
\frac{
160t^4-200t^3-1239t^2-891t-162
}{
216(t+1)^2(t+2)(t+3)
}.
\end{aligned}
\]
Again the denominator is positive.  To prove positivity of the numerator for
\(t\ge4\), write \(t=s+4\).  Then
\[
\begin{aligned}
160t^4-200t^3-1239t^2-891t-162
&=
160s^4+2360s^3+11721s^2  \\
&\qquad +20557s+4610,
\end{aligned}
\]
which is positive for every \(s\ge0\).  Hence the difference is positive for
every \(t\ge4\).  Therefore the \(n=4k+2\) upper bound fails for every
\(n=8t+6\) with \(t\ge4\).  The first example in this family is
\[
    C_4=T(4,4,5,4),
\]
which has \(38\) vertices.

\section{Concluding remark}

The examples above show that the proposed extremal trees in Conjecture 1.2 \cite{WDL2025} are
not globally extremal. Determining the exact maximum value of the Laplacian ratio among $n$-vertex trees remains open.

\bibliographystyle{plainurl}
\bibliography{references}

\end{document}